\documentclass{amsart}

\usepackage{graphicx}
\usepackage{amsthm}

\usepackage{fullpage}

\usepackage{tikz-cd}
\usepackage{amssymb}
\usepackage{amsmath}
\usepackage{biblatex} 
\usepackage{stmaryrd}
\usepackage{hyperref}
\usepackage{quiver}
\usepackage{bm}
\theoremstyle{plain}
\newtheorem{theorem}{Theorem}[section]
\newtheorem{proposition}[theorem]{Proposition}
\newtheorem{lemma}[theorem]{Lemma}
\newtheorem{corollary}[theorem]{Corollary}
\theoremstyle{definition}
\newtheorem{definition}{Definition}
\newtheorem{example}{Example}
\newtheorem{remark}[theorem]{Remark}

\title{Rank jumps for Jacobians of hyperelliptic curves on K3 surfaces}
\author{Ander Arriola Corpion and Cec\'ilia Salgado}
\date{}

\begin{document}

\maketitle

\begin{abstract}
We study Mordell--Weil rank jumps on families of jacobians of a pencil of genus-2 curves on a K3 surface defined over a number field $k$. We exhibit a finite extension $l/k$ over which the subset of fibers for which the rank jumps is infinite. Moreover, we describe further geometric conditions on the K3 surface under which the rank jumps on a non-thin set of fibers.
\end{abstract}
\section{Introduction}
\vspace{5pt}
Let $k$ be a number field, $S$ a smooth, geometrically integral curve over $k$, and $\eta: \mathcal{A} \rightarrow S$ a non-constant family of abelian varieties. Let $A$ be its generic fiber. Let $l/k$ be a finite field extension. The group of $l$-sections of $\eta$ is naturally isomorphic to the Mordell--Weil group $A(l(S))$. By the Lang--Néron theorem \cite[Theorem 6.1]{FundamentalsLang}, this group is finitely generated; we denote its rank by $r(\eta, l)$. For a point $s \in S(l)$, let $\mathcal{A}_s := \eta^{-1}(s)$ be the corresponding fiber and let $r_s$ denote the Mordell--Weil rank of $\mathcal{A}_s$ over $l$. A specialization theorem by Silverman, which builds on work of Néron, \cite{NeronSpecialization, SilvermanSpecialization} implies that the inequality
\begin{equation}\label{ineq:Silverman}
r_s \geq r(\eta, l)
\end{equation}
holds for all but finitely many $s \in S(l)$. The aim of this note is to study Mordell--Weil \textit{rank jumps} within these families, i.e., the set

\[
\mathcal{R}(\eta,l) \ = \ \{ \ s \in S(l) \mid r_s > r(\eta, l) \ \}
\]
of fibers for which the inequality in (\ref{ineq:Silverman}) is strict. Evidently, this is only interesting when $S(l)$ is infinite.

Our focus lies on families of jacobians associated to a pencil of genus-$2$ curves on a K3 surface $X$ defined over $k$. So, in our case, $S\cong\mathbb{P}^1$. K3 surfaces with a genus 2 curve admit a realization as a double cover $X \rightarrow \mathbb{P}^2_k$ branched along a plane sextic curve $B \subset \mathbb{P}^2_k$ \cite[Theorem~2.3]{DolgachevSpK3}. In this setting, our main results are as follows:

\begin{theorem}\label{thm1 intro}
Assume that $B$ contains a geometrically irreducible non-linear component defined over $k$.
Then, there is a field extension $l/k$ of degree at most 12 such that  
\[|\mathcal{R}(\eta,l)| = \infty.\]
\end{theorem}

The existence of a genus-$1$ fibration on $X$ often allows us to obtain rank jumps on a \emph{non-thin} set of fibers. A sufficient condition for this is the presence of a singularity on the branch sextic. 
More precisely, we obtain:

\begin{theorem}\label{thm2 intro}
Assume that $B$ admits a simple singularity defined over $k$. Then, there is a field extension $l/k$ of degree at most 18 such that $\mathcal{R}(\eta,l)$ is not thin.
\end{theorem}

\begin{remark} The field extension $l/k$ appearing in Theorems~\ref{thm1 intro} and~\ref{thm2 intro} depends on the geometry of the branch curve $B$ and on the field of definition of singular points in it. In many concrete cases, it can be chosen to be of smaller degree. For example, let $\mathcal{L}_d$ be the family of $U \oplus \langle -2d\rangle$-polarized K3 surfaces. By \cite[Proposition~3.16]{GarbagnatiSalgado} every $X \in \mathcal{L}_{2m+5}$ admits a realization as a double cover $X \rightarrow \mathbb{P}^2_k$ branched over a sextic with a simple node, and there exists a rational curve in $\mathbb{P}^2_k$ of degree $m+1$ passing through the node with multiplicity $m$ that splits in the double cover. So, in that particular case, the rank jump occurs over the base field $k$.
\end{remark}
\begin{corollary}\label{cor3:intro}
Let $X$ be a generic K3 surface in $\mathcal{L}_d$ with $d>3$, $d \equiv 3 $ (mod 4), and let $\eta: \mathcal{A} \rightarrow \mathbb{P}^1_k$ be the family of jacobians of a pencil of hyperelliptic curves of genus 2 in $X$. Then, $\mathcal{R}(\eta,k) \subset \mathbb{P}^1_k$ is not thin.
\end{corollary}

Finally, applying our results to a family of K3 surfaces of degree $2$  studied in~\cite{JuliaRatPtK3deg2} yields the following:

\begin{corollary} \label{cor:4 intro}
In the coarse moduli space $\mathcal{M}_2$ of K3 surfaces of degree 2 there is a dense family of polarized K3 surfaces $(X,\mathcal{L})$ defined over $\mathbb{Q}$ satisfying the following property: there is a dense set of curves in $ | \mathcal{L}|$ defined over $\mathbb{Q}$ with Jacobian of positive Mordell-Weil rank over $\mathbb{Q}$.\\
\end{corollary}

\subsection{Methodology and relation to the literature}
\mbox{}
\vspace{5pt}

Several authors have addressed the variation of the Mordell--Weil rank in families of abelian varieties. Different methods have been employed to treat the problem in the context of elliptic surfaces, namely the study of root numbers pioneered by Rohrlich (\cite{Rohrlich1993, VarillyAlvarado2011, Desjardins19}); the use of height theory as a tool in estimating the size of the set of fibers that witness a rank jump (\cite{BillardRepartition}); and the description of multisections that induce linearly independent sections after base change (\cite{SalgadoANT, LS, SalgadoPasten, CostaSalgado}). Our work  builds on the third method.\\

When $\eta \colon \mathcal{A} \rightarrow S$ is a family of abelian varieties defined over a number field $k$, with $S$ a smooth curve, Hindry and Salgado \cite{HindrySalgado} proved that the $k$-unirationality of $\mathcal{A}$ implies the existence of infinitely many fibers with a rank jump.  Colliot-Th\'elène \cite{ColliotTheleneRankJumps} later generalized this result, allowing $S$ to be higher dimensional and showing, in particular, that whenever there exists a variety $W$ with the Hilbert property and a dominant $k$-morphism $W \to \mathcal{A}$ such that the composite $W \to \mathcal{A} \to S$ has geometrically integral generic fiber, the rank jump set $\mathcal{R}(\eta,k)$ is not thin in $S$. \\

While our results deal with the special setting where $\mathcal{A} \rightarrow S$ is the family of jacobians of a pencil of curves on a K3 surface, we do not need the hypothesis that the underlying variety $\mathcal{A}$ is $k$-unirational nor the presence of a variety $W$ as described above. Indeed, we treat K3 surfaces that may possess at most one genus-$1$ fibration, or none at all. For such surfaces, the potential Hilbert property is generally unknown, and a suitable $W$ is not known to exist, in general, for the associated jacobian fibration.\\

K3 surfaces containing linear systems of hyperelliptic curves were studied systematically by Dolgachev \cite{DolgachevSpK3} and Reid \cite{ReidHypLinSysK3}, who obtained essentially equivalent classifications.  Dolgachev terms such surfaces \emph{special}, while Reid calls them \emph{hyperelliptic}.  In the terminology of \cite{DolgachevSpK3}, the present paper concerns pencils of hyperelliptic curves of genus $2$ on special K3 surfaces of \emph{class~$2$}.\\

Even when the underlying K3 surface is known to satisfy the potential Hilbert property, for instance, when it admits multiple genus-$1$ fibrations \cite{MezzedimiGvirtzchenNonThinRatPtsK3}, our approach provides an explicit bound on the degree of the field extension $l/k$ over which $\mathcal{R}(\eta,l)$ is not thin.  Such a bound does not follow directly from \cite{MezzedimiGvirtzchenNonThinRatPtsK3} and \cite{ColliotTheleneRankJumps}, as the former does not give an explicit bound on the degree of the field over which the Hilbert property holds.\\

Our method relies crucially on the notion of \emph{saliently ramified multisections}, introduced by Bogomolov and Tschinkel in the context of elliptic surfaces \cite{BogTschRPEF}. Pasten and Salgado \cite{SalgadoPasten} proved that after base change they become sections that are independent of the pullback of the original Mordell--Weil group, providing a tool to force rank jumps.  This technique was applied successfully in \cite{SalgadoPasten} and more recently in \cite{GarbagnatiSalgado}.  In our setting, we produce multisections of genus at most 1 for the pencil of genus-$2$ curves on the K3 surface. Under the hypothesis of Theorem~\ref{thm1 intro} the multisections are rigid, whereas in Theorem~\ref{thm2 intro} they move in a linear family.\\

\subsection{Organization of the text}
\mbox{}
\vspace{5pt}

This article is organized as follows: Section~\ref{sec: Background} recalls the necessary background on fibered surfaces and their jacobians, specialization theorems, multisections, and thin sets.  Section~\ref{sec: pencils on K3s} focuses on pencils of hyperelliptic genus-$2$ curves on K3 surfaces and contains the proofs of our main theorems.  Finally, Section~\ref{sec:examples} presents explicit examples of K3 surfaces with small Picard number to which our results apply. These examples are not known to satisfy the previously mentioned methods, so our constructions yield new families of abelian fibrations with the rank‑jump property.\\

\section{Background}\label{sec: Background}

Throughout this article, we let $k$ denote a number field. By a \emph{variety} we mean a separated scheme of finite type over $k$ that is geometrically reduced. A \emph{curve} (resp.\ a \emph{surface}) is a variety of pure dimension $1$ (resp.\ $2$). Unless explicitly stated otherwise, all varieties and morphisms are defined over the base field $k$.

For a field extension $l/k$ and a $k$-variety $X$, we write
\[
X_l := X \times_k \operatorname{Spec}(l)
\]
for the base change of $X$ to $l$. If $f : X \to Y$ is a morphism of $k$-varieties, we denote the induced morphism after base change by $f_l : X_l \to Y_l$.\\

\subsection{Fibered surfaces and Jacobians} \label{fibsurfandjac}
\mbox{}
\vspace{5pt}

The material of this subsection is standard. We follow Hassett's exposition (\cite[Section 4]{Hassett}) and present it here for the sake of fixing the notation for the rest of the paper. 

\begin{definition}
Let $S$ be a smooth, projective, geometrically integral curve. 
\begin{enumerate}
    \item An \textit{$S$-fibered variety} is a smooth, projective, geometrically integral variety $\mathcal{X}$ together with a dominant morphism $\pi: \mathcal{X} \rightarrow S$. If $S$ is clear from the context, we call it a \textit{fibered variety}.
    \item An \textit{$S$-fibered surface} is an $S$-fibered variety $\pi: \mathcal{X} \rightarrow S$ with $\dim \mathcal{X}=2$.
    \item An \textit{$S$-abelian fibration} is an $S$-fibered variety $\pi: \mathcal{A} \rightarrow S$ whose generic fiber is an abelian variety over the function field of $S$. An abelian fibration of relative dimension 1 is an \textit{elliptic fibration}.
\end{enumerate}
\end{definition}

Let $\pi: \mathcal{X} \rightarrow S$ be a fibered variety. For each closed point $s \in S$, we denote by $\mathcal{X}_s$ the fiber of $\pi$ over $s$ and by $\xi$ the generic point of $S$.  Since $\mathcal{X}$ is regular and integral, and $\mathrm{char} \ k=0$, the generic fiber $\mathcal{X}_\xi$ is a smooth integral variety over $k(S)$.

We define the smooth locus
\begin{equation} \label{smoothlocus}
    S^0:= \{ \ s \in S \ | \ \pi^{-1}(s)= \mathcal{X}_s \text{ is smooth} \ \}
\end{equation}
and denote by \[\pi^0: \mathcal{X}^0 \rightarrow S^0\] the restriction of $\pi$ to $S^0$. 

If the generic fiber $\mathcal{X}_\xi$ is geometrically connected (and hence geometrically integral), then $\pi_* \mathcal{O}_\mathcal{X}=\mathcal{O}_S$ and all the fibers of $\pi$ are geometrically connected. In particular, for every $s\in S^0$, the fiber $\mathcal{X}_s$ is a smooth, geometrically integral variety.

\begin{definition}
Let $\pi: \mathcal{X} \rightarrow S$ be a fibered variety. A geometrically integral curve $M \subset \mathcal{X}$ is called a \emph{multisection of degree $n$} if the restriction \[\pi_{|_M}: M \rightarrow S\] is a surjective and finite morphism of degree $n$.
\end{definition}

Let $\pi: \mathcal{X} \rightarrow S$ be an $S$-fibration, let $M\subset \mathcal{X}$ be a multisection of $\pi$ and let \[\hat{M} \rightarrow M\] be its normalization. The composition \[g: \hat{M} \rightarrow M \overset{\pi_{|_M}}{\longrightarrow} S\] is a finite morphism of smooth curves. Let $\mathcal{X}_{\hat{M}}$ be the minimal desingularization of the fiber product $\mathcal{X} \times_S \hat{M}$ and let $\pi_{\hat{M}} : \mathcal{X}_{\hat{M}} \rightarrow \hat{M}$ be the projection. The generic fiber of $\pi_{\hat{M}}$ is $$(\mathcal{X}_\xi)_{k(M)}=\mathcal{X}_\xi \times_{k(S)} k(M),$$
so $\pi_{\hat{M}}: \mathcal{X}_{\hat{M}} \rightarrow {\hat{M}}$ is again a fibered variety. Since $\pi^0:\mathcal{X}^0 \rightarrow S^0$ is smooth, its base change \[\pi_{\hat{M}}^{-1}(\hat{M}^0) \cong \mathcal{X}^0 \times_{S^0} \hat{M}^0 \ \rightarrow \hat{M}^0, \, \,\, \, \, \,  \, \hat{M}^0:=g^{-1}(S^0)\] is also smooth. In particular, \[ \mathcal{X}^0_{\hat{M}^0}:=\pi_{\hat{M}}^{-1}(\hat{M}^0)\cong \mathcal{X}^0 \times_{S^0} \hat{M}^0 \] is a smooth open $k$-subvariety of $\mathcal{X}_{\hat{M}}$.

Each section $\sigma : S \rightarrow \mathcal{X}$ induces a section of $\mathcal{X} \times_S \hat{M} \rightarrow \hat{M}$ as illustrated in the following fiber diagram:
\[\begin{tikzcd}
	\hat{M} \\
	S & {\mathcal{X} \times_S \hat{M}} & \hat{M} \\
	& {\mathcal{X}} & S
	\arrow[from=1-1, to=2-1]
	\arrow["\exists"{description}, from=1-1, to=2-2]
	\arrow["{\mathrm{Id}}"', curve={height=-12pt}, from=1-1, to=2-3]
	\arrow["\sigma"', from=2-1, to=3-2]
	\arrow[from=2-2, to=2-3]
	\arrow[from=2-2, to=3-2]
	\arrow["\lrcorner"{anchor=center, pos=0.125}, draw=none, from=2-2, to=3-3]
	\arrow[from=2-3, to=3-3]
	\arrow[from=3-2, to=3-3]
\end{tikzcd}\]
This yields a rational section $\hat{M} \dashrightarrow \mathcal{X}_{\hat{M}}$. Since $\hat{M}$ is smooth and $\mathcal{X}_{\hat{M}}$ is projective, it extends uniquely to a morphism \[\sigma_{\hat{M}}: \hat{M} \rightarrow \mathcal{X}_{\hat{M}}\] which is a section of $\pi_{\hat{M}}$. In this way, we obtain an injection \[\mathcal{X}(S) \rightarrow \mathcal{X}_{\hat{M}}(\hat{M}).\]

From now on, we assume that $\pi: \mathcal{X} \rightarrow S$ is a fibered surface whose generic fiber $X:=\mathcal{X}_\xi$ is a smooth, geometrically connected curve of genus $g$ over $k(S)$. Then there is a natural correspondence
\begin{equation} \label{msecptcorresp}
\begin{array}{c c c}
     \{ \ \text{closed points of } X \ \text{of degree } n \}& \leftrightarrow & \{ \ \text{multisections of } \pi \text{ of degree } n \}\\
     P & \mapsto & \overline{P} \subset \mathcal{X}\\
     M \cap X& \mapsfrom & M
\end{array}    
\end{equation}
see \cite[Proposition 8.3.4]{LiuAG} for a proof.

Under this correspondence, the inclusion $\mathcal{X}(S) \subset \mathcal{X}_{\hat{M}}(\hat{M})$ constructed above coincides with the natural inclusion \[X(k(S)) \subset X(k(M))=X_{k(M)}(k(M)).\] 

The restriction $\pi^0: \mathcal{X}^0 \rightarrow S^0$ is a smooth family of irreducible curves. We consider the  relative Jacobian $$\eta^0:\mathcal{J}^0=\mathrm{Pic}^0_{\mathcal{X}^0/S^0} \rightarrow S^0$$ 
which is an abelian scheme over $S^0$ (\cite[Theorem and Definition 27.137]{gwAG2}). The morphism $\eta^0$ is projective and since $S^0$ is quasi-projective, so is $\mathcal{J}^0$. For each $s \in S^0$, the fiber $\mathcal{J}^0_s$ is the Jacobian of $\mathcal{X}_s$.

The relative Jacobian $\eta^0 : \mathcal{J}^0 \rightarrow S^0$ admits a smooth compactification to an abelian fibration \[\eta: \mathcal{J} \rightarrow S.\]

We denote by \[J:=\mathcal{J}_\xi=\mathrm{Jac}(X)\] the jacobian of the generic fiber.
\indent\par
\indent\par
Let $K:=k(S)$, and fix an algebraic closure $\overline{K}$ of $K$. Let $O\in \mathcal{X}_\xi(K)$, and let \[ \sigma_O: S \rightarrow \mathcal{X}\] be the corresponding section, so that $\sigma_O(S)=\overline{O}$. By \cite[Corollary 27.143]{gwAG2} we obtain a closed immersion 
\[j^0_O: \mathcal{X}^0 \hookrightarrow \mathcal{J}^0\]
which on each fiber coincides with the Abel-Jacobi map \[j^0_{O,s}: \mathcal{X}^0_s \hookrightarrow \mathcal{J}^0_s, \, \, \, \, x \mapsto[x-O_s],\] where $O_s=\overline{O} \cap \mathcal{X}_s$. We write \[j_O:=j^0_{O,\xi}:X \hookrightarrow J\] for the induced embedding on the generic fiber.
\indent\par
\indent\par
Let $P \in X$ be a closed point of degree $n$, and let $M:= \overline{P} \subset \mathcal{X}$ be the corresponding multisection of $\pi$  (\ref{msecptcorresp}). Let 
\[ 
\hat{M} \rightarrow M \subset \mathcal{X} 
\] be its normalization. Denote by $K(P)$ the residue field of $P$. Notice that $K(P)=k(M)$. The closed point $P$ corresponds to a $\mathrm{Gal}(\overline{K}/K)$-orbit of points \[\{P_1, \ldots, P_n\} \subset X(\bar{K})\] associated with embeddings
\[ \alpha_1, \ldots , \alpha_n: K(P) \hookrightarrow \overline{K}.
\]

Set $K_i:=\alpha_i(K(P))$. Identifying $K(P)$ with $K_1=\alpha_1(K(P))$, we may assume that $K(P) \subset \overline{K}$ and $\alpha_1=\mathrm{Id}_{K(P)}$. Each point $P_i$ determines a closed embedding
\[j_i : X_{K_i} \hookrightarrow J_{K_i}, \, \, \, x \mapsto[x-P_i]\] defined over $K_i$.
\indent\par
\indent\par
For each $i=1, \ldots ,n$, the base change of the closed embedding $\mathrm{Spec}(K(P)) \hookrightarrow X$ to $K_i$ yields a degree-$n$ zero cycle in $X_{K_i}$, hence a $K_i$-point of $\mathrm{Sym}^nX_{K_i}$. Its image under the natural morphism 
\[ 
\mathrm{Sym}^nX_{K_i} \rightarrow J_{K_i}
\]
induced by $P_i$ is
\begin{equation} \label{neosections}
    s^i_M:=(P_1+ \cdots + P_n)-nP_i \in J(K_i).
\end{equation}
\indent\par
Let $F$ be the normal closure of $K(P)$ in $\overline{K}$. Then all fields $K_i$ are contained in $F$ and the points $s^1_M, \ldots,s_M^n$ are permuted by the action of $\mathrm{Gal}(F/K)$. In particular, they all have the same order in $J(F)$. 

Finally, for all $1 \leq i,j \leq n$, we have 
\begin{equation} \label{relation1}
 s^i_M-s^j_M=n(P_j-P_i)=n \cdot j_i(P_j) \in J(F)   
\end{equation}
and for every $1 \leq i \leq n$,
\begin{equation} \label{relation2}
   s^i_M= \sum_{j=1}^nP_j-nP_i = \sum_{j=1}^n j_i(P_j) \in J(F). 
\end{equation}

\subsection{Abelian fibrations and specialization}
\mbox{}
\vspace{5pt}

Let $\eta : \mathcal{A} \rightarrow S$ be an abelian fibration over $k$ and let $A$ be its generic fiber. The $k(S)$-points of $A$ can be identified with the $k$-sections of $\eta$ and for $s \in S^0(k)$ we have the specialization homomorphism:
$$
\begin{array}{c c c c}
    \mathrm{sp}_s : & A(k(S))& \longrightarrow & \mathcal{A}_s(k)\\
     & \sigma & \longmapsto & \sigma(s).
\end{array}
$$  

\begin{theorem} \label{specialization}
For all but finitely many $s \in S^0(k)$ the specialization homomorphism $\mathrm{sp}_s$ is injective.
\end{theorem}

The theorem follows from the specialization Theorem \cite[Theorem~C]{SilvermanSpecialization} by Silverman when the generic fiber of $\eta : \mathcal{A} \rightarrow S$ has a trivial $k(S)/k$-trace. We show that this hypothesis is not needed. We use the following result by Manin from \cite{pTorsManin}. The version below is from \cite[Theorem~5.2.1]{SerreLecturesMW}.
\begin{theorem}[Djamanenko-Manin] \label{djamanenkomanin}
Let $A$ be an abelian variety and $X$ a projective, normal curve, both defined over a number field $k$. For a fixed $x_0 \in X(k)$, let
$$A_{x_0}(X):=\{ \ f \in \mathrm{Hom}_k(X,A) \ | \ f(x_0)=0_A \}.$$
Let $\{f_1, \ldots , f_r \}$ be a set of free generators of $A_{x_0}(X)$. For all but finitely many $x \in X(k)$, the points $f_1(x), \ldots , f_r(x) \in A(k)$ are $\mathbb{Z}$-linearly independent.
\end{theorem}

By \cite[Proposition~3.1]{HindrySalgado}, if the genus of $S$ is 0, then the trivial trace hypothesis is not needed. For the sake of completion, we reproduce the argument here.

\begin{proof}[Proof of Theorem \ref{specialization}]
The statement is trivially true if $S(k)$ is finite so by Falting's theorem \cite[Satz 7]{Faltings1983}, we may assume that the genus of $S$ is either 0 or 1 and $|S(k)|=\infty$. Let $A$ be the generic fiber of $\eta: \mathcal{A} \rightarrow S$, an abelian variety over $K:=k(S)$. Assume that the $K/k$-trace $T:= \mathrm{Tr}_{K/k} \ A$ of $A$ is not trivial.

The specialization map is always injective on torsion points and to deal with the non-torsion points we may replace $A$ by an isogenous abelian variety. If $T$ is nontrivial, the abelian variety $A$ is isogenous to a product $B \times_{K} T_K$ for some abelian variety $B$ with trivial trace, so we may assume that $A=B \times_K T_K$. Since  the result is true for trivial trace families \cite{SilvermanSpecialization}, we can reduce to the case when the family $\eta : \mathcal{A} \rightarrow S$ is a constant family, that is, $A=B \times_k S$ for some abelian variety $B$ over $k$.

In this case, observe that sections $S \rightarrow B \times_k S$ correspond to morphisms $S \rightarrow B$. If the genus of $S$ is $0$, that is, $S \cong \mathbb{P}^1_k$, then all morphisms $S \cong \mathbb{P}^1_k \rightarrow B$ are constant  and the claim is trivially true. If the genus of $S$ is $1$, that is, if $S$ is an elliptic curve, every morphism $S \rightarrow B$ corresponds to a homomorphism of abelian varieties $S \rightarrow B$ sending $0_S \mapsto 0_B$, that is, following the notation of Theorem \ref{djamanenkomanin},
$$
\begin{array}{c c c}
     (B \times_k S)(S) & \overset{\cong}{\longrightarrow}& B_{0_S}(S)  \\
     s \mapsto (\phi(s),s) \ & \longmapsto & t_{-\phi(0_S)} \circ \phi
\end{array}
$$
where $t_P$ for $P \in B(k)$ denotes the ``translation by $P$" morphism. In this case, the claim follows by Theorem \ref{djamanenkomanin}.

\end{proof}

Given a finite field extension $l/k$, the group $A(l(S))$ is isomorphic to the group of sections of $\mathcal{A}_l \rightarrow S_l$ and by the Lang-Néron theorem \cite[Theorem~6.1]{FundamentalsLang} it is finitely generated. The \emph{generic Mordell--Weil rank} of $\eta$ (over $l$) is the rank of $A(l(S))$ and we denote it by $r(\eta,l)$. Theorem \ref{specialization} gives
$$\operatorname{rank} \mathcal{A}_s(l) \geq r(\eta,l) $$
for all but finitely many $s \in S^0(l)$. The set of \emph{rank jumps (over $l$)} is defined as
\begin{equation} \label{eq:rankjumpset}
    \mathcal{R}(\eta,l) := \bigl\{ s \in S^0(l) \mid \operatorname{rank}\mathcal{A}_s(l) > r(\eta,l) \bigr\}.
\end{equation}

Finally, if $\eta : \mathcal{A} \rightarrow S$ arises as the jacobian fibration of a fibration of curves $\pi \colon \mathcal{X} \rightarrow S$, we write $\mathcal{R}(\pi,l)$ for $\mathcal{R}(\eta,l)$ by abuse of notation.\\

\subsection{Multisections of fibrations and rank jumps}
\mbox{}
\vspace{5pt}

Let $\eta: \mathcal{A} \rightarrow S$ be an abelian fibration defined over $k$ and let $M$ be a multisection of $\eta$. Following the notation established in subsection \ref{fibsurfandjac}, let $\hat{M} \rightarrow M$ be the normalization of $M$ and  $\mathcal{A}_{\hat{M}}$ the minimal desingularization of the base change $\mathcal{A}\times_S \hat{M}$. The morphism $\eta$ induces a fibration $\eta_{\hat{M}}: \mathcal{A}_{\hat{M}} \rightarrow \hat{M}$. Since every section of $\eta$ pulls-back to a section of $\eta_{\hat{M}}$, the following holds
\begin{equation}\label{eq:generic rank multisection}
r(\eta_{\hat{M}},k) \geq r(\eta,k).
\end{equation}
Let
$$g:\hat{M}\rightarrow M\overset{\eta|_M}{\longrightarrow} S$$
be the composite morphism. By Theorem \ref{specialization} applied to $\eta_{\hat{M}}$, for all but finitely many $t \in \hat{M}(k)$,
\begin{equation}\label{eq: specialization multisection}
\operatorname{rank}(\mathcal{A}_{g(t)}(k))=\operatorname{rank}\big((\mathcal{A}_M)_t(k) \big) \geq r(\eta_{\hat{M}},k) \geq r(\eta,k),
\end{equation}
where the equality follows from the isomorphism $\mathcal{A}_{g(t)} \cong (\mathcal{A}_M)_t$.

Our aim is to produce situations where the inequality~\eqref{eq:generic rank multisection} is strict. To that end, we use the notion of \emph{saliently ramified multisections} introduced by Bogomolov and Tschinkel~\cite[Definition~4.3]{BogTschRPEF}. 

\begin{definition}\label{def: saltiently ramified}
Let $\pi: \mathcal{X} \rightarrow S$ be a fibered variety and let $M \subset \mathcal{X}$ be a multisection. Denote by $\hat{M} \rightarrow M$ the normalization of $M$.
We say that $M$ is saliently ramified (with respect to $\pi$) if the composite morphism
\[
\hat{M} \xrightarrow{\nu} M \overset{\pi|_M}{\longrightarrow} S
\]
is ramified over the open subset $S^0 \subset S$.
\end{definition}

\begin{remark}
Definition \ref{def: saltiently ramified} differs slightly from \cite[Definition 4.3]{BogoTschinkelPotDensityEllK3}, we allow $M$ to be singular. For $M$ to be saliently ramified, a smooth fiber of $\pi$ may meet $M$ with non‑reduced intersection at a smooth point of $M$, while $M$ may be singular elsewhere. 
\end{remark}

Let $\pi: \mathcal{X} \rightarrow S$ be a fibered surface and let 
$$\eta: \mathcal{J} \rightarrow S$$ 
be a smooth compactification of the relative jacobian ${\eta^0:\mathcal{J}^0=\mathrm{Pic}^0(\mathcal{X}^0/S^0) \rightarrow S^0}$. The morphism $\eta^0:\mathcal{J}^0 \rightarrow S^0$ is an abelian scheme \cite[Theorem/Definition~27.137]{gwAG2}. Since $\mathrm{char} \ k =0$, the multiplication-by-$N$ isogeny 
$$[N]:\mathcal{J}^0 \rightarrow \mathcal{J}^0$$  
is finite and étale for every integer $N \in \mathbb{Z}$ \cite[Proposition~27.187]{gwAG2} and consequently, its kernel $\mathrm{ker} \ [N] \rightarrow S^0$ is a finite étale group scheme over $S^0$.

\begin{lemma} \label{rkjplemma}
Let $M\subset \mathcal{X}$ be a saliently ramified multisection of $\mathcal{X} \overset{\pi}{\rightarrow} S$ and let $K=k(S)$, $L=k(M)$ be the corresponding function fields. Then,
$$\mathbb{Z} \cdot s^1_M \cap J(K)=\{ 0 \}$$
in $J(L)$. In particular, the Mordell-Weil rank of $J(L)$ is strictly larger than the Mordell-Weil rank of $J(K)$. 
\begin{proof} 
We follow the notation introduced in (\ref{neosections}). Let $n$ be the degree of the multisection $M$ and let $F$ be the normal closure of $L$ in $\overline{K}$. Assume, for contradiction, that there exists a nonzero integer $m \in \mathbb{Z}$ such that $[m](s_M^1)=Q \in J(K)$. Since the points $s_M^1, \ldots , s_M^n$ are $\mathrm{Gal}(F/K)$-conjugate, we obtain $m\cdot s^i_M=Q$ for all $1 \leq i \leq n$. By (\ref{relation1}), for each $1 \le j \le n$, the following holds in $J(F)$:
\[mn(j_1(P_j)) = m(nP_j-nP_1)=m(s^j_M-s^1_M)=Q-Q=0.\]
 Hence
 \[j_1(P_j) \in \mathrm{ker} \ [mn]. \] 

Consider the closed subscheme $\mathrm{Spec}(L \otimes_K L) \subset X_L$ which defines a zero-cycle of degree $n$. The $\operatorname{Gal}(F/L)$-orbits of the set ${P_1,\ldots,P_n} \subset X(F)$ correspond bijectively to the closed points of $\mathrm{Spec}(L \otimes_K L)$ whose degree is given by the size of the corresponding orbit. Hence the image of this $0$-cycle under $j_1$ is contained in $\ker[mn] \subset J_L$.

We extend this picture over the open subschemes $S^0$ and $\hat{M}^0$. Let \[j^0_1: \mathcal{X}^0_{\hat{M}^0} \hookrightarrow \mathcal{J}^0_{\hat{M}^0}\] be the extension of $j_1$ over $\hat{M}^0$. 

Since $M$ is saliently ramified, the morphism $\hat{M}^0 \to S^0$ is ramified, and consequently the fiber product
\[
\hat{M}^0 \times_{S^0} \hat{M}^0 \rightarrow \hat{M}^0
\]
is also ramified. On the other hand, by the previous discussion, 
\[ j^0_1(\hat{M}^0 \times_{S^0} \hat{M}^0) \subset \ker[mn].\]
This is impossible, since $\ker[mn] \to \hat{M}^0$ is étale. Therefore such nonzero $m$ cannot exist, and 
\[\mathbb{Z} \cdot s^1_M \cap J(K)=\{ 0 \}.\]
The final assertion on the ranks follows immediately.

\end{proof}
\end{lemma}

Our goal is to study the set rank jump set $\mathcal{R}(\eta,k)$ for the jacobian fibration $\eta:\mathcal{J} \rightarrow S$ associated with a fibration of curves $\pi:\mathcal{X} \rightarrow S$. Recall that by a slight abuse of notation we write $\mathcal{R}(\pi,k):=\mathcal{R}(\eta,k)$. 

As a direct consequence of Lemma \ref{rkjplemma}, we obtain the following key result.
 
\begin{proposition} \label{infptrkjump}
Let $\pi: \mathcal{X} \rightarrow S$ be a fibered surface over $k$ and let $ M \subset \mathcal{X}$ be a saliently ramified multisection of $\pi$. Then 
\[\pi(x) \in \mathcal{R}(\pi,k), \text{ for all but finitely many } x \in M(k).\]  

In particular, if $|M(k)|=\infty$, then $|\mathcal{R}(\pi,k)|=\infty$.
\begin{proof}
Let $\eta: \mathcal{J} \rightarrow S$ be the Jacobian fibration of $\pi$ and let $\nu \colon \hat{M} \to M$ be the normalization of $M$. By Lemma~\ref{rkjplemma}, the generic rank of the base-changed fibration
\[\eta_{\hat{M}} : \mathcal{J}_{\hat{M}} \rightarrow {\hat{M}}\] 
is strictly greater than the generic rank of $\eta$, that is,
\begin{equation} \label{rkstrictineq}
    r(\eta_{\hat{M}},k) > r(\eta,k).
\end{equation}
The specialization Theorem \ref{specialization} gives that the set
\[ \mathcal{S}:=\big\{ \ t \in \hat{M}(k) \ | \ \operatorname{rank} (\mathcal{J}_{\hat{M}})_t \geq r(\eta_{\hat{M}},k) \ \big\}\]
contains all but finitely many $t \in \hat{M}(k)$. Let $g := \pi \circ \nu$. For every $t \in \hat{M}^0(k)$, there is a canonical identification
$(\mathcal{J}_{\hat{M}})_t \cong \mathcal{J}_{g(t)}$
and therefore
\[g ( \mathcal{S} ) \subset  \mathcal{R}(\eta,k)\]
by the inequality \eqref{rkstrictineq}.  Since $\nu : \hat{M} \to M$ is birational, the claim follows.

\end{proof}
\end{proposition}

When the generic fiber of $\pi: \mathcal{X} \rightarrow S$ is a smooth, geometrically integral curve of genus 1, we recover \cite[Theorem 2.11]{BogoTschinkelPotDensEnriques}.

\begin{corollary} \label{cor: g1manyposrankfibers}
Let $\pi \colon \mathcal{X} \to S$ be a fibered surface over $k$, whose generic fiber is a smooth, geometrically connected curve of genus~$1$. Let $M \subset \mathcal{X}$ be a saliently ramified multisection of $\pi$. Then, \[|\mathcal{X}_{\pi(x)}(k)|=\infty \text{ for all but finitely many } x \in M(k).\] 
In particular, if $M(k)$ is infinite, then $\mathcal{X}(k)$ is Zariski dense in $\mathcal{X}$.\\
\end{corollary}

\subsection{Thin subsets}
\mbox{}
\vspace{5pt}

We use Serre’s definition of thin sets \cite[Sec. 3.1]{SerreLecturesMW}.
\begin{definition}
Let $V$ be a variety over a field $k$. A subset $Z\subseteq V(k)$ is called \emph{thin} if it is contained in a finite union of subsets of $V$ of the following two types:
\begin{itemize}
    \item[(1)] proper closed subvarieties of $V$;
    \item[(2)] subsets of the form $\varphi(W(k))$, where $\varphi : W \to V$ is a generically finite dominant morphism of degree at least $2$, with $W$ an integral $k$-variety.
\end{itemize}
\end{definition}

As we are interested in infinite subsets of $\mathbb{P}^1(k)$, for $k$ a number field, we are concerned with subsets of type 2 in what follows.

\begin{definition}
A variety $V$ is said to satisfy the \textit{Hilbert property over $k$} if $V(k)$ is not a thin subset of $V$.
\end{definition}

The most prominent examples of varieties with the Hilbert property over $\mathbb{Q}$, or more generally number fields, are projective spaces, notably equivalent to Hilbert's irreducibility theorem (\cite[Chapter~9]{SerreLecturesMW}).\\

\section{Pencils of hyperelliptic curves of genus two on K3 surfaces}\label{sec: pencils on K3s}
\vspace{5pt}

Let $X$ be a K3 surface over a field $k$ of characteristic 0. If $X$ contains a hyperelliptic curve $C \subset X$ of genus 2, then the complete linear system $|C|$ is base point free and induces a generically finite morphism of degree two 
\[f: X \rightarrow \mathbb{P}^2_k\]
branched along a plane sextic $B \subset \mathbb{P}^2_k$ with at most simple singularities. Conversely, if $X$ is a K3 surface that admits a generically finite morphism of degree two $f: X \rightarrow \mathbb{P}^2_k$, the pullback of the complete linear system of lines in $\mathbb{P}^2_k$ gives a linear system in $X$ whose generic element is a hyperelliptic curve of genus 2.

Via Stein factorization we may write $f$ as 
\[X \overset{\mu}{\rightarrow} \bar{X} \overset{\bar{f}}{\rightarrow} \mathbb{P}^2_k\] 
where $\bar{f}$ is a finite morphism of degree 2 and $\mu$ is a birational morphism. The sextic $B$ is the zero locus of a homogeneous polynomial $F(x,y,z)$ of degree 6 and the surface $\overline{X}$ is given by in the weighted projective space $\mathbb{P}(1,1,1,3)$ by
\[w^2=F(x,y,z).\]

The surface $\bar{X}$ is singular exactly above the singular points of $B$ and $\mu: X \rightarrow \bar{X}$ is the minimal resolution of those singularities. 

 In what follows, we  denote by $B_{\mathrm{sing}}$ the singular locus of $B$ and by $B_{\mathrm{reg}}=B \setminus B_{\mathrm{sing}}$ the regular locus.\\

\subsection{The pull-back of lines} \label{subsection: pblines}
\mbox{}
\vspace{5pt}

We describe the curves in $X$ that arise as pull-backs of lines in $\mathbb{P}^2_k$ by the double cover $f$. As we are interested in a geometric description, we assume in what follows that $k$ is algebraically closed. 

Let $\ell \in |\mathcal{O}_{\mathbb{P}^2_k}(1)|$ be a line and let $C_\ell$ denote the strict transform of $\bar{f}^{-1}(\ell)$ in $X$. If $\ell \subset B$, then \[C_\ell = 2R,\] with $R\subset X$ isomorphic to $\ell$. 

Now assume that $\ell \not\subset B$. For $x \in \ell \cap B$ denote by $i_x(\ell,B)$ the intersection multiplicity of $\ell$ and $B$ at $x$. By Bézout's theorem, 
\[6=(\deg \ell) (\deg B)= \sum_{x \in \ell \cap B} i_x(\ell,B).\]
If $i_x(\ell \cap B)=1$, then $\bar{f}^{-1}(x)$ is a smooth point of the curve $\bar{f}^{-1}(\ell) \subset \bar{X}$. If $i_x(\ell,B)=n \geq 2$, then $\bar{f}^{-1}(\ell)$ has an $A_{n-1}$ singularity at $\bar{f}^{-1}(x)$. 

Since $\ell \not\subset B$, the divisor $\bar{f}^{-1}(\ell)$ is reduced. Because $\bar{f}$ is a double cover, exactly one of the following holds:
\begin{itemize}
    \item[a)]  $\bar{f}^{*} \ell$ is irreducible, if $i_x(\ell, B)$ is odd for some $x$;
    \item[b)] $\bar{f}^* \ell=R_1 + R_2$ for two rational curves $R_1$ and $R_2$, if $i_x(\ell, B)$ is even for every $x\in B$.
    \end{itemize}
In case (b) the line $\ell$ is called a \emph{tritangent line} of $B$.

In what follows, we focus on the subcases of (a) that are relevant to our applications.

\begin{lemma}\label{lemma: pull back line}
Let $\ell$ be a line in $\mathbb{P}^2$. Assume that $\ell$ is not tritangent to $B$. Then, the following hold:
\begin{itemize}

\item[i)] If $|B \cap \ell |=6$ then $C_\ell =f^{-1}(\ell) \cong \bar{f}^{-1}(\ell)$ is a smooth hyperelliptic curve of genus 2.

\item[ii)] If $|B \cap \ell |=5$ then $C_\ell$ is a geometrically integral curve of geometric genus 1.

\item[iii)] If $|B \cap \ell |=4$ then $C_\ell$ is a geometrically integral curve of geometric genus at most 1; 

\end{itemize}
\begin{proof}
In all cases, geometric integrality of $C_\ell$ follows from the fact that we are in the situation a) above.\\

The condition $|B \cap \ell |=6$ forces $i_x(\ell,B)=1$ for all $x \in \ell \cap B$, so $C_\ell$ is smooth. The intersection $B \cap \ell$ is precisely the branch locus of the morphism $f_{|_{C_\ell}}: C_\ell \rightarrow \ell \cong \mathbb{P}^1_k$ so the genus of $C_\ell$ is 2 by the Riemann-Hurwitz formula.\\

If $|B \cap \ell |=5$, exactly one point $y \in \ell \cap B$ satisfies $i_y(\ell,B)=2$ and $i_x(\ell,B)=1$ for the rest of points $x \in (\ell \cap B) \setminus \{ y \}$. Thus $\bar{f}^{-1}(\ell)$ has a node at $f^{-1}(y)$ and is smooth elsewhere. Let $E$ be the normalization of $\bar{f}^{-1}(\ell)$. The composition 
$$E \rightarrow \bar{f}^{-1}(\ell) \rightarrow \ell \cong \mathbb{P}^1_k$$ 
is branched over $(\ell \cap B) \setminus \{ y \}$. By the Riemann-Hurwitz formula, the genus of $E$ is 1.\\

If $|B \cap \ell |=4$, two possibilities can occur: either there are two points $y_1, y_2 \in \ell \cap B$ with $i_{y_1}(\ell,B)=i_{y_2}(\ell,B)=2$ and $i_x(\ell,B)=1$ for the remaining points $x \in (\ell \cap B) \setminus \{ y_1, y_2 \}$, or $i_y(\ell,B)=3$ for one $y \in \ell \cap B$ and $i_x(\ell , B)=1$ for the remaining $x \in (\ell \cap B) \setminus \{y\}$. In the former, $\bar{f}^{-1}(\ell)$ has two nodes and, in the latter, it has a cusp. By the Riemann-Hurwitz formula, the normalization of $\bar{f}^{-1}(\ell)$ has geometric genus 0 or 1, respectively.

\end{proof}
\end{lemma}

For the remainder of the section, we return to the assumption that $k$ is a number field. Fix an algebraic closure $\overline{k}$ of $k$.

Fix a point $P=[x_0:y_0:z_0] \in \mathbb{P}^2(k)$ and let $L_P\subset |\mathcal{O}_{\mathbb{P}^2_k}(1)|$ be the pencil of lines through $P$. We study the pullback pencil $f^*L_P$. A direct application of Lemma \ref{lemma: pull back line} gives the following description, which depends on $P$.

\begin{proposition}\label{prop: pull back pencil}
Let $P\in \mathbb{P}_k^2(k)$. The generic member of the moving part of the pencil $f^*L_P$ is:
\begin{enumerate}
    \item a smooth curve of genus 1, if $P$ is a singular point of $B$, or
    \item a smooth hyperelliptic curve of genus 2, otherwise.
 \end{enumerate}
\end{proposition}

Let $H_P$ be the moving part of the pullback pencil $f^*L_P$, and let
\[
    \phi_{L_P}: \mathbb{P}^2_k \dashrightarrow \mathbb{P}^1_k, \ \text{and} \  \phi_{H_P}: X \dashrightarrow \mathbb{P}^1_k 
  \]
    
    be the rational maps induced by the pencils $L_P$ and $H_P$, respectively. If $P \in B_{\mathrm{sing}}(k)$, the base locus of $H_P$ is empty so $\phi_{H_P}$ is actually a morphism giving a genus-1 fibration on $X$. If $P \notin B_{\mathrm{sing}}(k)$, after a blow-up of the base locus of $H_P$, we obtain a fibered surface \begin{equation}\label{eq: genus 2 fibration}\pi_P: \mathcal{X}_P \rightarrow \mathbb{P}^1_k
    \end{equation} 
whose generic fiber is a smooth hyperelliptic curve of genus 2. The situation is summarized by the following commutative diagram:
\begin{equation} \label{diagramblowup}
    \begin{tikzcd}
	& {\mathcal{X}_P} \\
	& X & {\mathbb{P}^1_k} \\
	{\bar{X}} & {\mathbb{P}^2_k}
	\arrow["\nu_P"', from=1-2, to=2-2]
	\arrow["\pi_P", from=1-2, to=2-3]
	\arrow["{{\phi_{H_P}}}"'{pos=0.4}, dashed, from=2-2, to=2-3]
	\arrow["\mu"', from=2-2, to=3-1]
	\arrow["f"', from=2-2, to=3-2]
	\arrow["{\bar{f}}"', from=3-1, to=3-2]
	\arrow["{{\phi_{L_P}}}"', dashed, from=3-2, to=2-3]
\end{tikzcd}
\end{equation}
where $\nu_P$ is the blow-up of the base locus of $H_P$, and recall that $\bar{X} \subset \mathbb{P}(1,1,1,3)$ is the surface given by 
$$w^2=F(x,y,z).$$
When $P\in B_{\mathrm{sing}}(k)$, we also write  
$\pi_P:=\phi_{H_P}$, $\mathcal{X}_P:=X$, and $\nu_P:=\mathrm{id}$ to unify notation.

Given a line $\ell \subset \mathbb{P}^2_k$ , recall that $C_\ell \subset X$ is the strict transform of $\bar{f}^{-1}(\ell)$ by the blowup $\mu: X \rightarrow \bar{X}$. In what follows, we denote by $ \tilde{C_\ell}$ the strict transform of $C_\ell$ under the second blowup $\nu_P: \mathcal{X}_P \rightarrow X$. 

\begin{lemma} \label{pblinesmultsecram}
Let $\ell \subset \mathbb{P}^2_k$ be a line that is not a component of $B$. Assume that $P \notin \ell$ and that $\ell$ is not tritangent to $B$. Then, $\tilde{C}_\ell \cong C_\ell$ is a bisection of $\pi_P: \mathcal{X}_P \rightarrow \mathbb{P}^1_k$. 

Moreover, let $\hat{C}_\ell \rightarrow \tilde{C}_\ell$ be the normalization of $\tilde{C}_\ell$ and define
$$I:= \{ \ x \in \ell \cap B \ | \ i_x(\ell, B) \text{ is odd } \} \subset \mathbb{P}^2_k.$$
Then the branch locus of the composition $\hat{C}_\ell \rightarrow \tilde{C}_\ell \overset{\pi_P}{\longrightarrow} \mathbb{P}^1_k$ is $\phi_{L_P}(I) \subset \mathbb{P}^1_k$.
\begin{proof}
Since $\ell$ is not contained in $B$ and is not a tritangent line, the pullback $\bar{f}^{-1}(\ell) \subset \bar{X
}$ is geometrically integral, and the same holds for $C_\ell$ and $\tilde{C}_\ell$. Moreover, since $\ell$ does not pass through $P$ the restriction of $\nu_P$ to $C_\ell$ is an isomorphism so $C_\ell \cong \tilde{C}_\ell$. 

The curve $\tilde{C}_\ell$ is not a fiber of $\pi_P$ because $f^{-1}(\ell)\notin H_P$; consequently, $\pi_P(\widetilde{C}_\ell)=\mathbb{P}^1_k$, so $\tilde{C}_\ell$ is a multisection of $\pi_P$. 

We claim that the degree of the restriction
$$\pi_{P|_{\tilde{C}_\ell}}: \tilde{C}_\ell \rightarrow \mathbb{P}^1_k$$ 
is 2. Let $F$ be a smooth fiber of $\pi_P$, that is, $F= \tilde{C}_{\ell'}$ for a line $\ell' \subset \mathbb{P}^2_k$ through $P$ with $|(\ell \cap B)(\bar{k})|=6$. Then, 
$$F \cdot \tilde{C}_{\ell} = C_{\ell'} \cdot C_{\ell}=2,$$
so $\deg (\pi_{P|_{\tilde{C}_\ell}}) =2$. 

Finally, from diagram (\ref{diagramblowup}) the morphism $\hat{C}_\ell \rightarrow \mathbb{P}^1_k$ coincides with the composition: 
\begin{equation} \label{composition}
    \hat{C}_\ell \rightarrow\tilde{C}_{\ell} \overset{\nu_P}{\rightarrow} C_\ell \overset{\mu}{\rightarrow} \bar{f}^{-1}(\ell) \overset{\bar{f}}{\rightarrow} \ell \overset{\phi_{L_P}}{\rightarrow} \mathbb{P}^1_k.
\end{equation}
The composition of the first three maps, $\hat{C}_\ell \rightarrow \bar{f}^{-1}(\ell)$, gives the normalization of $\bar{f}^{-1}(\ell)$, while $f^{-1}(\ell) \rightarrow \ell$ is a finite morphism of degree 2 branched over $\ell \cap B$. By Lemma \ref{lemma: pull back line}, the curve $\bar{f}^{-1}(\ell)$ has an $A_{n-1}$ singularity above a point $x \in \ell \cap B$ for $n:=i_x(\ell, B) \geq 2$. The normalization $\hat{C}_\ell \rightarrow \bar{f}^{-1}(\ell)$ removes a ramification point on an $A_{n-1}$ singularity if and only if $n-1$ is odd, that is, when $n=i_x(\ell , B)$ is even. The last map $\ell \rightarrow \mathbb{P}^1_k$ in (\ref{composition}) is an isomorphism, which completes the proof.

\end{proof}
\end{lemma} 

\begin{remark}
In \cite[Section~4]{XiaoGangGenreDeux}, Gang classifies relatively minimal genus‑$2$ fibrations on complex algebraic surfaces with ``small'' numerical invariants.  
The fibrations $\pi_P\colon\mathcal{X}_P\to\mathbb{P}^1_k$ constructed above fall into case \textit{(5)} of Theorem 4.5 there.
\end{remark}

\subsection{Rank jumps for pencils of hyperelliptic curves of genus 2}
\mbox{}
\vspace{5pt}

We continue with the notation and setup of the previous subsection. Fix a point $P \in \mathbb{P}^2(k)$. We construct saliently ramified multisections for the fibration 
$$\pi_P : \mathcal{X}_P \rightarrow \mathbb{P}^1_k.$$
We present two constructions, namely the pull-back of a tangent line to $B$ (rigid), and the pull-back of a line passing through a singularity of $B$ (yielding a genus 1 fibration).

The first construction is largely inspired by the tangent correspondence from \cite[Definition~3.2]{BogTschRPEF}.

\begin{lemma} \label{lemmatglines}
Assume that $B_{\bar{k}}$ contains an irreducible component $B_0$ that is not a line. Then, for all but finitely many $x \in B_0$, the tangent line $T_x \subset \mathbb{P}^2_{\bar{k}}$ to $B_0$ at $x$ satisfies:
\begin{enumerate}
    \item $|T_x \cap B_{\bar{k}}|=5$, and
    \item there exists $y \in T_x \cap B_{\bar{k}}$ with $y \neq x$ such that $\phi_{L_P}(y) \in (\mathbb{P}^1_{\bar{k}})^0$.
\end{enumerate}
\begin{proof}
The set $\phi_{L_P}^{-1}\big( \mathbb{P}^1_k \setminus(\mathbb{P}^1_{k})^0 \big)$ consists of lines through $P$ that correspond to non-smooth fibers of $\pi_P$. Set 
$$Z := \phi_{L_P}^{-1}\big( \mathbb{P}^1_k \setminus(\mathbb{P}^1_{k})^0 \big) \cap B .$$
Then, if a line through $P$ is contained in $B$, it is also contained in $Z$ since it gives a non-reduced fiber of $\pi_P$. Since $P $ is at most a triple point of $B$, there are at most three such lines (over $\bar{k}$). Hence, we can write $Z_{\bar{k}}=R \cup Z'$, where $R$ is a union of at most three lines, and $Z'$ is a finite union of points. Therefore, only finitely many $x \in B_0$ satisfy the following: 
$$T_x \cap Z' \not= \emptyset, \ \ T_x \subset R \ \text{ or } \ |T_x \cap B_{\bar{k}} | < 5.$$ 
For all remaining $x \in B_0$ the tangent line $T_x$ of $B_0$ at $x$ satisfies properties (1) and (2) of the statement.

\end{proof}
\end{lemma}

In what follows, we show that tangent lines that satisfy the hypothesis of Lemma \ref{lemmatglines} yield saliently ramified bisections of $\pi_P$.

\begin{theorem} \label{theorem: tgmultisection}
Assume that $B$ contains a geometrically irreducible component $B_0$ of degree $d > 1$ defined over $k$. Then, there exists an extension $k'/k$ of degree at most $d$  over which $\pi_P$ admits a saliently ramified bisection of geometric genus 1. 
\begin{proof}
Since $\deg B_0=d>1$, intersection with $k$-lines yield the existence of infinitely many points on $B_0$ defined over a field of degree at most $d$.

By Lemmas \ref{lemma: pull back line} \ref{pblinesmultsecram} and \ref{lemmatglines}, there is a field extension $k'/k$ of degree at most $d$ and a point $x\in B_0(k')$ such that

\[M:=\tilde{C}_{T_x} \subset (\mathcal{X}_P)_{k'},\]
 is a geometrically integral curve of geometric genus 1 which is, moreover, a saliently ramified bisection of $\pi_P$.

\end{proof}
\end{theorem}

Theorem \ref{theorem: tgmultisection} combined with Proposition \ref{infptrkjump} yields:

\begin{corollary} \label{cor: rkjptangent}
Assume that $B$ contains a non-linear, geometrically irreducible component. Then, there exists a field extension $l/k$ of degree at most 12 such that 
$$|\mathcal{R} (\pi_P,l )|=\infty.$$
\begin{proof}
Let $M$ and $k'$ be the multisection and the field extension as in Theorem \ref{theorem: tgmultisection}, respectively. Let $\hat{M} \rightarrow M$ be the normalization of $M$. The composition 
$$\hat{M} \rightarrow M \rightarrow T_x \cong \mathbb{P}^1_{k'}$$ 
realizes $\hat{M}$ as a 2:1 cover of $\mathbb{P}^1_{k'}$ branched over four points, and \cite[Proposition 4.2]{JuliaRatPtK3deg2} guarantees the existence of infinitely many quadratic extensions $l/k'$ with $|\hat{M}(l)|=\infty$. Therefore, $|M(l)|=\infty$ as well. The claim follows from Proposition \ref{infptrkjump}.
\end{proof}
\end{corollary}

\vspace{5pt}

\begin{remark} \label{rem: infptlowdeg}
Clearly, if the geometrically irreducible component $B_0$ of degree $d>1$ contains infinitely many points of degree $d'<d$, the extension $k'/k$ in \ref{theorem: tgmultisection} can be chosen of degree at most $d'$. For example, if $|B_0(k)|=\infty$, we may choose $k'=k$ and hence $|\mathcal{R}(\pi_P,l)|=\infty$ over an extension $l/k$ of degree at most 2. 
\end{remark}

\vspace{5pt}

When $B$ is singular, we apply  the strategy introduced in \cite{SalgadoPasten} and obtain a stronger result for elliptic K3 surfaces. Indeed, in this case, the pullback of the pencil of lines through a singularity of $B$ yields a pencil of genus 1 curves in $\mathcal{X}_P$ whose members are saliently ramified multisections of $\pi_P: \mathcal{X}_P \rightarrow \mathbb{P}^1_k$. We use this family to show that, after a suitable base change $l/k$, the set $\mathcal{R}(\eta,l)$ is not only infinite  but is moreover not thin in $\mathbb{P}^1_l$.
\\

Let $Q \in B_{\mathrm{sing}}(k)$ be a singular point defined over $k$. We first identify an extension $l/k$ over which the pencil of genus-1 curves
$$\pi_Q : X_l \rightarrow \mathbb{P}^1_l$$
contains infinitely many curves, each having infinitely many $l$-points. To achieve this, we find a saliently ramified multisection for $\pi_Q$ and apply Corollary~\ref{cor: g1manyposrankfibers}. One way to obtain such a multisection is by using Theorem \ref{theorem: tgmultisection}. Alternatively, if $X$ admits another genus 1 fibration (for instance, when $B$ has several singular points) we may use that fibration to generate multisections for $\pi_Q$, following the method of \cite{SalgadoPasten}.

\begin{lemma} \label{lemma: twosingularities}
Assume that $B$ admits two singular points, $Q,Q' \in B_{\mathrm{sing}}(k)$, and that at least one of them is a double point. Then, for all but finitely many $s \in \mathbb{P}^1(k)$ the fiber 
\[E_s:=\pi_{Q'}^{-1}(s)\] 
is a saliently ramified multisection of $ \pi_Q: X \rightarrow \mathbb{P}^1_k $.
\begin{proof}
Let $U_Q$ and $U_{Q'}$ be the loci over which $\pi_Q$ and $\pi_{Q'}$ have smooth fibers, respectively. Let $Z$ be the union of the lines through $Q$ not contained in $B$ that 
\begin{itemize}
    \item have a tangent direction to $B$ at $Q$, or
    \item intersect $B$ with multiplicity greater than 1 at some point distinct from $Q$.
\end{itemize}
Then $Z$ is a union of lines not contained in $B$ so the intersection $Z \cap B$ is finite. Set \[V:=U_{Q'} \setminus \phi_{L_{Q'}}(Z \cap B).\]
For each $s \in V(k)$, let $\ell_s=\phi_{L_{Q'}}^{-1}(s)$, and $E_s:=\pi_{Q'}^{-1}(s)=C_{\ell_s}$.

Then, if $Q'$ is a double point, $\ell_s$ meets $B$ in four points away from $Q'$. At most three of them lie on lines through $Q$ contained in $B$, because $Q$ is at most a triple point. 

If $Q$ is a double point, $\ell_s$ meets $B$ in at least three points away from $Q'$. At most two of them lie on lines through $Q$ contained in $B$, because $Q$ is a double point.  

By Lemma \ref{pblinesmultsecram}, $E_s$ is a saliently ramified multisection of $\pi_Q$ in both cases, as claimed.

\end{proof}
\end{lemma}

In what follows, we construct multisections for $\pi_Q: X \rightarrow \mathbb{P}^1_l$. We apply Theorem \ref{theorem: tgmultisection} if $B$ has non-linear irreducible components, and Lemma \ref{lemma: twosingularities} if $B$ is (geometrically) a union of lines.

\begin{proposition} \label{prop: rkjumpelliptic}
Assume that $B$ has a singular point $Q \in B_{\mathrm{sing}}(k)$ and let ${\pi_Q:X \rightarrow \mathbb{P}^1_k}$ be the genus 1 fibration induced by the pencil of lines through $Q$. Then, there is a field extension $l/k$ of degree at most 18 such that the set
\[\big\{ \ t \in \mathbb{P}^1(l) \ \big| \ |\pi_Q^{-1}(t)(l)|=\infty \ \big\} \]
is infinite. 
\begin{proof}
We distinguish cases according to the decomposition of $B_{\bar{k}}$ in irreducible components.
\vspace{5pt}

\textbf{Case 1: $B_{\bar{k}}$ is a union of six lines.}
Since $Q$ is a simple singularity, it lies on at most three of these lines. We want to find a suitable extension of $k$ over which a second intersection point of the six lines is defined and then apply Lemma~\ref{lemma: twosingularities} to construct a saliently ramified multisection for $\pi_Q$. We distinguish two subcases depending on the number of lines through $Q$:

\vspace{3pt}
\noindent
\textit{Case 1.1: $Q$ is a intersection of exactly two lines.}
Let $\ell_1$ and $\ell_2$ be the two lines that meet in $Q$. Then there is an  extension $k_0/k$ of degree at most 2 over which they are defined. Over $k_0$ we can write 
$$B_{k_0}=\ell_1 \cup \ell_2 \cup R$$ 
where $R$ is a geometrically reducible quartic with four linear components over $\bar{k}$. Therefore, there is an extension $k'/k_0$ of degree at most four for which there is a point
$$Q' \in (\ell_1 \cap R)(k').$$

\vspace{3pt}
\noindent
\textit{Case 1.2: $Q$ is a intersection of three lines.}
Let $\ell \subset \mathbb{P}^2_{\bar{k}}$ be one of these three lines. Then $\ell$ is defined over an extension $k_0/k$ of degree at most 3. Let $\ell'$ be a linear component of $B_{\bar{k}}$ that does not pass through $Q$. We can choose $\ell'$ such that $\ell '\cap \ell$ is a double point $Q'$ of $B$ (i.e., no other component of $B$ passes through $Q'$). Moreover, $\mathrm{Gal}(\bar{k}/k)$ permutes the components of $B$ that do not pass through $Q$. Therefore, there is an extension $k_1/k$ of degree at most 3 over which $\ell'$ is defined. Let $k':=k_0\cdot k_1$ be the compositum of $k_0$ and $k_1$, then $Q'$ is defined over $k'$. 
\\
\\
Let $Q$ be as in \textit{Case 1.1 or 1.2.} Then, by Lemma~\ref{lemma: twosingularities}, there is a smooth genus 1 curve $E \subset X_{k'}$ defined over $k'$ that is a saliently ramified multisection of $\pi_Q$. By \cite[Proposition 4.2]{JuliaRatPtK3deg2}, there is a further quadratic extension $l/k'$ with $|E(l)|=\infty$, and Corollary \ref{cor: g1manyposrankfibers} gives the result.

%We can choose an extension $k'/k$ of degree 9 such that $\ell, \ \ell'$ and $Q'$ are defined, and apply Lemma~\ref{lemma: twosingularities} to obtain a smooth genus 1 curve $E \subset X_{k'}$ that is a saliently ramified multisection of $\pi_Q$. Again, \cite[Proposition~4.2]{JuliaRatPtK3deg2} provides a quadratic extension $l/k'$ such that $|E(l)|=\infty$, and Corollary \ref{cor: g1manyposrankfibers} gives the result.

\vspace{5pt}
\textbf{Case 2: $B_{\bar{k}}$ is not a union of six lines.}
Let $B_0$ be a component of  $B_{\bar{k}}$ of degree $d>1$. Then, $B_0$ is defined over an extension $k_0/k$ of degree $[k_0:k]\leq \frac{6}{d}$. By Theorem \ref{theorem: tgmultisection} there is a further extension $k'/k_0$ of degree at most $d$ together with a saliently ramified multisection $M \subset X_{k'}$ of geometric genus 1. By \cite[Proposition 4.2]{JuliaRatPtK3deg2}  there is a quadratic extension $l/k'$ with $|M(l)|=\infty$. A direct application of Corollary \ref{cor: g1manyposrankfibers} completes the proof.

\end{proof}
\end{proposition}

Finally, we show that over the field extension $l/k$ provided by Proposition \ref{prop: rkjumpelliptic}, the set $\mathcal{R} ( \pi_P,l)$ is not thin in $\mathbb{P}^1_l$.

\begin{theorem} \label{theorem: nonthinrkjumpgenus2}
Assume that $B$ has a  $k$-rational singular point $Q \in B_{\mathrm{sing}}(k)$ and that $P \notin B_{\mathrm{sing}}(k)$. Then, there exists a field extension $l/k$ of degree at most 18 such that the set $\mathcal{R}(\pi_P,l) \subset \mathbb{P}^1_l$ is not thin.
\begin{proof}
Let $\pi_Q: X \rightarrow \mathbb{P}^1_k$ be the genus 1 fibration induced by the pencil of lines through $Q$, and let $U_P$, $U_Q$ denote the loci of smooth fibers of $\pi_P$, $\pi_Q$, respectively. 

By Proposition \ref{prop: rkjumpelliptic}, there is a field extension $l/k$ of degree at most 18 such that $\pi_Q$ has infinitely many fibers over $l$ with infinitely many $l$-points. We base change to $l$ and drop the subscript $l$ to simplify the notation.

\vspace{3pt}
In what follows, we show that $\mathcal{R}(\pi_P, l)$ is not thin in $\mathbb{P}^1_l$. We follow the strategy of \cite[Theorem 1.1]{SalgadoPasten}.

 Let $\varphi_i: Y_i \rightarrow \mathbb{P}^1_l$ ($i=1, \ldots,m$) be morphisms of smooth geometrically irreducible curves of degree $\deg \varphi_i \geq 2$ and $|Y_i(l)|=\infty$. The following diagram illustrates the situation:
 \[\begin{tikzcd}
	& {\mathcal{X}_P} && \\
	{\mathbb{P}^1_l} & X & {\mathbb{P}^1_l} & {Y_i} \\
	& {\mathbb{P}^2_l}
	\arrow["{\nu_P}"', from=1-2, to=2-2]
	\arrow["{\pi_P}", from=1-2, to=2-3]
	\arrow["{\pi_Q}"', from=2-2, to=2-1]
	\arrow["{\phi_{H_P}}", dashed, from=2-2, to=2-3]
	\arrow["f"{pos=0.3}, from=2-2, to=3-2]
	\arrow["{\varphi_i}"', from=2-4, to=2-3]
	\arrow["{\phi_{L_Q}}", dashed, from=3-2, to=2-1]
	\arrow["{\phi_{L_P}}"', dashed, from=3-2, to=2-3]
\end{tikzcd}\]
with notation as introduced in \eqref{diagramblowup}.
 
 Denote by  $\Sigma \subset \mathbb{P}^1_l$ the union of all branch points of the $\varphi_i$. Recall that for $s \in \mathbb{P}^1_l$, the preimage $\ell_s:=\phi_{L_P}^{-1}(s)$ is a line through $P$ defined over $l$. Set:
\[ \Omega :=\{ \ s \in \mathbb{P}^1_l(l) \ | \ \ell_s \not\subset B_l \text{ and } \exists x \in \ell_s \cap B_l \text{ with }i_x(\ell_s, B_l)>1 \ \}\]
and define
\[Z:= \bigcup_{\substack{s \in \Sigma(l) \cup \Omega\\ \ell_s \not \subset B}} \ell_s  \subset \mathbb{P}^2_l.\]
Since $Z$ is a finite union of lines not contained in $B_l$, the intersection $Z\cap B_l$ is finite. Now set
\[V:=U_Q \setminus \phi_{L_Q}(Z \cap B_l) \subset \mathbb{P}^1_l.\]
By Lemma \ref{pblinesmultsecram}, for every $t \in V(l)$, the curve $E_t:=\tilde{C}_{\phi_{L_Q}^{-1}(t)} \cong \pi_Q^{-1}(t)$ is a saliently ramified bisection of $\pi_P: \mathcal{X}_P \rightarrow \mathbb{P}^1_l$. Moreover, for each $i$, there is at most one line through $P$ contained in $B$, so the branch loci of $\pi_{P|_{E_t}}$ and $\varphi_i$ share at most one point. Moreover, infinitely many $t \in V(l)$ satisfy $|E_t(l)|=\infty$. Fix one such $t_0$ and set $E:=E_{t_0}$.

Let $l(\mathbb{P}^1_l)$, $l(E)$ and $l(Y_i)$ denote the function fields of $\mathbb{P}^1_l$, $E$ and $Y_i$, respectively. Because $[l(E):l(\mathbb{P}^1_l)]=2$, for each $i$ either $l(E)\cap l(Y_i)=l(\mathbb{P}^1_l)$ or $l(E)\subset l(Y_i)$.   The latter would force $\varphi_i$ to factor through $\pi_P|_E$, which is impossible because the branch locus of $\pi_P|_E$ is not contained in the branch locus of $\varphi_i$.  
Hence $l(E)/l(\mathbb{P}^1_l)$ and each $l(Y_i)/l(\mathbb{P}^1_l)$ are linearly disjoint and the fiber product $E \times_{\mathbb{P}^1_l} Y_i$ is geometrically integral for every $i=1, \ldots , m$. Let 
$D_i \rightarrow E \times_{\mathbb{P}^1_l} Y_i$ 
be the normalization. We have the following diagram:
\[\begin{tikzcd}
	{D_i} & {E \times_{\mathbb{P}^1_l} Y_i} & {Y_i} \\
	& E & {\mathbb{P}^1_l}
	\arrow[from=1-1, to=1-2]
	\arrow[from=1-2, to=1-3]
	\arrow[from=1-2, to=2-2]
	\arrow["{\varphi_i}", from=1-3, to=2-3]
	\arrow["{\pi_{P|_E}}"', from=2-2, to=2-3]
\end{tikzcd}\]
Since ${\pi_P}|_E$ and $\varphi_i$ share at most one branch point, the composition $D_i \rightarrow E$ is ramified, and therefore $D_i$ has genus at least 2. By Falting's theorem \cite{Faltings1983}[Satz 7], each $D_i(l)$ is finite, whereas $E(l)$ is infinite, so 
\[E(l) \setminus \bigcup_{i=1}^n\psi_{i}(D_i(l))\] 
is infinite. 

Finally, by Proposition \ref{infptrkjump}, for all but finitely many $x \in E(l)$ we have $\pi_P(x) \in \mathcal{R}(\pi_P,l)$. Therefore,
\[\mathcal{R}(\pi_P,l) \setminus \bigcup_{i=1}^n \varphi_i(Y_i(l))\]
is infinite, and hence $\mathcal{R}(\pi_P,l)$ is not thin.

\end{proof}
\end{theorem}

\begin{remark} \label{rem: choiceoffield1}
In specific situations the degree bounds can often be improved.
\begin{itemize}

\item[a)] The field extension $l/k$ is chosen so that the genus 1 fibration $\pi_Q$ has infinitely many fibers over $l$ with infinitely many $l$-rational points. Sometimes this is achieved already over the base field $k$. For example, let $\mathcal{L}_d$ be the family of $U \oplus \langle-2d\rangle$-polarized K3 surfaces studied in \cite{GarbagnatiSalgado}. If $d>3$ and $d \equiv 3 $ (mod 4) this is true for a generic member in $\mathcal{L}_d$, see \cite[Proposition~3.16~and~Theorem~5.3]{GarbagnatiSalgado}. This gives Corollary \ref{cor3:intro}.

\item[b)] If $B$ has an irreducible component $B_0$ of degree $\deg B_0 >1 $ with infinitely many $k$-rational points, we can always choose an extension $l/k$ of degree at most 2 (see Remark \ref{rem: infptlowdeg}).
\end{itemize}
\end{remark}

\vspace{7pt}
\begin{remark}
If $B$ is smooth, the surface $X$ might also admit a genus 1 fibration $\epsilon:  X \rightarrow \mathbb{P}^1_k$ (e.g. \cite[Example~5.8]{VanGeemenBrauerGroupsK3}) and we obtain a result analogous to Theorem \ref{theorem: nonthinrkjumpgenus2}.  For example, in the situation described in \cite[Example~5.8]{VanGeemenBrauerGroupsK3}, we obtain a non-thin rank jump over an extension $l/k$ such that $[l:k]\leq 6$.  Indeed, the cover involution yields a second genus 1 fibration $\bar{\epsilon}:X \rightarrow \mathbb{P}^1_k$ whose fibers are, all but finitely many, saliently ramified smooth multisections of degree 9 for $\epsilon$. Fix a smooth $\bar{\epsilon}^{-1}(t_0)$.  There is an extension $l/k$ of degree at most 3 over which $\bar{\epsilon}^{-1}(t_0)$ has a point and, after a further extension of degree at most 2, it has positive Mordell--Weil rank. Corollary \ref{cor: g1manyposrankfibers} yields the desired result.  %We apply the same argument as in the proof of Theorem \ref{theorem: nonthinrkjumpgenus2} to obtain non-think rank jumps for $\pi_P$ over $l$.  
We observe that  \cite[Theorem~1.1]{MezzedimiGvirtzchenNonThinRatPtsK3} and \cite[Theorem~1.1]{ColliotTheleneRankJumps} can be applied to show the non-thin rank jump over a finite field extension $l/k$ since the K3 above is doubly elliptic. However, this would not yield a bound on the degree $[l:k]$.

%there exists a field extension $l/k$ such that $X(l)$ is not thin, and \cite[Theorem~1.1]{ColliotTheleneRankJumps} implies that $\mathcal{R} \big( \pi_P,l \big)$ is not thin for every $P \in \mathbb{P}^2(l)$. However, \cite[Theorem~1.1]{MezzedimiGvirtzchenNonThinRatPtsK3} does not provide an effective bound for the degree of $l/k$.
%
%Alternatively, one can try to produce multisections for $\pi$ using tangent lines to $B$ as in Theorem~\ref{theorem: tgmultisection}, and then employ $\pi$ as a varying family of multisections for $\pi_P$.
%{\textcolor{red}{this second paragraph above is not good. Alternatively to what? it is not an alternative as the field is semi-explicit in this situation. Also "try to" is not okay. and what do you mean with "employ". it's vague. Alternative below, but I suggest we delete this part and all the remark below.}} 

%More generally, for a smooth branch sextic with a genus 1 fibration, our methods yield the bound $[l:k]\leq 6$ in this situation. Indeed, the tangent lines to $B$ yield saliently ramified bisections for the genus 1 fibration defined over an extension of degree at most 6 ({\textcolor{red}{check: 6 or 12.}}), analogously to the situation in Theorem \ref{theorem: tgmultisection}. Hence infinitely many genus 1 fibers have positive rank over $k'$. The remaining argument is as in the proof of Theorem \ref{theorem: nonthinrkjumpgenus2}.
\end{remark}

\subsection{Rank jumps in the complete linear system}
\mbox{}
\vspace{5pt}

In what follows, we study Mordell–Weil rank jumps in the complete linear system $H = |f^*\mathcal{O}_{\mathbb{P}^2_k}(1)|$. Consider the universal family 
$$\pi:\mathcal{C} \rightarrow H \cong \mathbb{P}^2_k,$$
where $\mathcal{C}$ is the incidence variety of the curves in $H$. Let $C$ denote the generic fiber of $\pi$, which is a hyperelliptic curve of genus 2 over $k(H) \cong k(\mathbb{P}^2_k)=k(t,s)$, and let $J=\mathrm{Jac}(C)$ be its jacobian over $k(s,t)$. For every finite extension $l/k$, the group $J(l(t,s))$ is finitely generated by the Lang-Néron theorem \cite{FundamentalsLang}[Theorem 6.1]. Let $r(\pi,l)$ denote its rank. Let $H^0 \subset H$ be the smooth locus of $\pi$. In the same spirit as for pencils in $H$, we seek a finite extension $l/k$ over which the set
$$\mathcal{R}(\pi,l)= \{ \ D \in H^0(l) \ | \ \mathrm{rank(}\mathrm{Jac}(D)(l)) > r(\pi,l) \ \}$$
is dense in $H_l$, or better, not thin in $H_l$.\\

Each point $P \in \mathbb{P}^2(k)$ determines a pencil of curves $H_P$ in $H$, with notation as in \ref{subsection: pblines}. If $P \notin B_{\mathrm{sing}}$, via the isomorphism $H_P \cong \mathbb{P}^1_k$, the smooth locus $\pi_P^0 : \mathcal{X}^0_P \rightarrow (\mathbb{P}^1_k)^0$ of the fibration $\pi_P$ constructed in subsection \ref{subsection: pblines} can be identified with the pullback $\pi^{-1}(H_P^0) \rightarrow H_P^0$, where $H_P^0 := H_P \cap H^0$. 

We first compare the ranks $r(\pi,k)$ and $r(\pi_P,k)$. By Wazir's generalization of Silverman's specialization theorem to higher-dimensional families \cite{WazirSpecialization}, for all but finitely many $P \in \mathbb{P}^2(k)$,
\[r(\pi_P,k) \geq r(\pi,k).\]
In particular, with a slight abuse of notation (identifying $H_P \cong \mathbb{P}^1_k$),
$$\mathcal{R}(\pi_P,k) \subset \mathcal{R}(\pi,k)$$
for all but finitely many $P \in \mathbb{P}^2(k)$. 

\begin{corollary} \label{cor: rkjmpcompletesystem}
With the notation above,
\begin{enumerate}
    \item If $B$ contains a $k$-point $x \in B(k)$ that is not in a line tritangent to $B$, then there is a field extension $l/k$ of degree at most 2 such that $\mathcal{R}(\pi,l)$ is dense in $H_l \cong \mathbb{P}^2_l$.
    \item If $B$ has a singular point $Q \in B_{\mathrm{sing}}(k)$, then there exists a field extension $l/k$ of degree at most 18 such that $\mathcal{R}(\pi,l)$ is not thin in $H_l \cong \mathbb{P}^2_l$ .
\end{enumerate}
\begin{proof}
Let us first show (1). The pullback $M=f^{-1}(T_x) \subset X$ is an irreducible  singular curve of geometric genus at most 1, and its normalization $E=\hat{M}$ is a genus 1 curve or a rational curve, see Lemma \ref{lemma: pull back line}. By Proposition~4.2 of \cite{JuliaRatPtK3deg2} there is a quadratic extension $l/k$ with $|E(l)|=\infty$, and hence $|M(l)|=\infty$ as well. We base‑change to $l$ and drop the subscript $l$ to simplify the notation.

Let $M$ be as above. There exists open set $U \subset \mathbb{P}^2_l$ such that for all $P \in U(l)$, the strict transform of $M$ by the blowup $\nu_P : \mathcal{X}_P \rightarrow X$ is a saliently ramified bisection of the fibration 
$\pi_P$. 
By Proposition \ref{infptrkjump},
$$|\mathcal{R}(\pi_P,l)|=\infty \text{    for all } P \in U(l).$$

Assume, for contradiction, that $\mathcal{R}(\pi,l)$ is not dense in $H \cong \mathbb{P}^2_l$, that is, there exist finitely many curves $C_1, \ldots , C_n \subset H_l$ over $l$ such that 
$$\mathcal{R}(\pi,l) \subset C_1(l) \cup \cdots C_n(l).$$ 
Then, for all but finitely many $P \in \mathbb{P}^2(l)$, the intersection $\mathcal{R}(\pi,l) \cap H_P$ is finite. However, for all but finitely many $P \in U(l)$, we have $\mathcal{R}(\pi_P,l) \subset \mathcal{R}(\pi,l) \cap H_P$ (via the isomorphism $H_P \cong \mathbb{P}^1_{l}$) and $\mathcal{R}(\pi_P,l)$ is infinite, which contradicts our assumption.\\

We verify (2). Let $l/k$ be the field provided by Theorem \ref{theorem: nonthinrkjumpgenus2}. Again we base change to $l$ and drop the subscript $l$. Assume, for contradiction, that $\mathcal{R}(\pi,l) \subset H \cong \mathbb{P}^2_l$ is thin. Then, there is a dense, open subset $U \subset \mathbb{P}^2_l$ such that for every $P \in U(l)$ the intersection $H_P \cap \mathcal{R}(\pi,l)$ is thin, see \cite[p.~127]{SerreLecturesMW}. However, for all but finitely many $P\in U(l)$ we have $\mathcal{R}(\pi_P,l) \subset \mathcal{R}(\pi,l)\cap H_P$, and $\mathcal{R}(\pi_P,l)$ is not thin by Theorem~\ref{theorem: nonthinrkjumpgenus2}, which clearly contradicts our assumption.
\end{proof}
\end{corollary}

\begin{remark}
Since there are infinitely many points in $B(\bar{k})$ of degree at most 6, we can always find an extension $k'/k$ of degree at most six and a point $x \in B(k')$ satisfying the hypotheses of Corollary \ref{cor: rkjmpcompletesystem} (1).

The field extension in part (2) is provided by Theorem \ref{theorem: nonthinrkjumpgenus2}. As noted in Remark \ref{rem: choiceoffield1} a)  the degree bound can sometimes be lowered, for example when $B$ contains a non‑linear geometrically irreducible component with infinitely many $k$-rational points.\\
\end{remark}

A very general K3 surface of degree 2 defined over $\mathbb{Q}$ satisfies the hypotheses of Corollary \ref{cor: rkjmpcompletesystem}[(1)].

\begin{proof}[Proof Corollary \ref{cor:4 intro}]
In \cite{JuliaRatPtK3deg2}[Theorem~1.3], Martínez-Marín constructs a family of sextics $T$ defined over $\mathbb{Q}$ with the following properties: 
\begin{itemize}
    \item[i)] For every $B \in T$ the double cover $f_B: X_B \rightarrow \mathbb{P}^2_{\mathbb{Q}}$ branched over $B$ is a degree two K3 surface.
    \item[ii)] Let $\mathcal{M}_2$ be the coarse moduli space of K3 surfaces of degree 2 and let $\psi:T \rightarrow \mathcal{M}_2$ be the projection taking each $B \in T$ to the isomorphism class of $X_B$. The image $\psi(T)$ is dominant. 
    \item[iii)] For every $B \in T$, we find $x=[1:0:0] \in B(\mathbb{Q})$ and the tangent line $\ell$ of $B$ at $x$ satisfies $|(\ell \cap B)(\bar{\mathbb{Q}})|=5$. Furthermore, $M=f^{-1}(\ell)$ is a nodal genus 1 curve with infinitely many $\mathbb{Q}$-points.
\end{itemize}
The claim follows from Corollary \ref{cor: rkjmpcompletesystem}. Notice that we can choose $l=\mathbb{Q}$ because $|M(\mathbb{Q})|=\infty$.

\end{proof}

\section{Examples}\label{sec:examples} 
In what follows, we present examples of K3 that satisfy the hypothesis of at least one of our results and that are not known to have the potential Hilbert property. Moreover, the underlying abelian fibration given by the jacobians of the genus 2 pencil is not known to satisfy the hypotheses in Th\'eor\`eme 3.3 in \cite{ColliotTheleneRankJumps}. Therefore, a priori, the results of \cite{ColliotTheleneRankJumps} cannot be applied to them.

For completeness, we include a brief background on the Shioda-Tate formula, which is used repeatedly in our examples.
\vspace{5pt}

\subsection{The Shioda-Tate formula for higher genus fibrations.} \label{subsectionshioda}
\mbox{}
\vspace{5pt}

Let $k$ be a number field and $\pi: \mathcal{X} \rightarrow S$ a fibered surface over $k$, whose generic fiber $X:=\mathcal{X}_{\xi}$ is a smooth, geometrically connected genus $g$ curve. Let $J=\mathrm{Jac}(X)$ be the Jacobian of $X$ over the function field $k(S)$. Let us assume that $\pi$ has a section over $\bar{k}$ that is, $O\in X(\bar{k}(S)) \neq \emptyset$.\\

In what follows, we introduce the main tool to compute the rank of $J(\bar{k}(S))$: the Shioda-Tate formula (cf. \cite{ShiodaHG}). Let $N \subset \mathrm{NS}(\mathcal{X}_{\bar{k}})$ be the subgroup of the (geometric) N\'eron-Severi group generated by:
\begin{enumerate}
    \item the class of the section $\bar{O} \subset \mathcal{X}_{\bar{k}}$,
    \item the class of a fiber $F$ of $\pi$,
    \item components of the reducible fibers of $\pi$ (over $\bar{k}$) that do not intersect $\bar{O}$.
\end{enumerate}

\begin{theorem}[Shioda-Tate] \label{theorem: shioda formula}
Let $(T, \tau)$ be the $\bar{k}(S)/\bar{k}$-trace of $J$. There is a natural isomorphism
\begin{equation} \label{ShiodaIso}
    J(\bar{k}(S))/\tau T(\bar{k}) \cong \mathrm{NS}(\mathcal{X}_{\bar{k}})/N.
\end{equation}
In particular, if $T$ is trivial, the following formula holds:
\begin{equation} \label{ShiodaFormula}
    \operatorname{rank} \ J(\bar{k}(S)) = \rho(\mathcal{X}_{\bar{k}})-2- \sum_{s \in S(\bar{k})} (m_s -1),
\end{equation}
where $m_s$ is the number of components of the fiber $\pi^{-1}(s)$ that do not intersect $\bar{O}$ and $\rho(\mathcal{X}_{\bar{k}})$ is the rank of the N\'eron-Severi group of $\mathcal{X}_{\bar{k}}$.
\end{theorem}

\begin{remark}
\mbox{}

\begin{enumerate}
    \item By \cite[Theorem 3]{ShiodaHG}, the triviality of the trace of $J$ is equivalent to the irregularity of $\mathcal{X}$ being equal to the genus of $S$.
    \item The isomorphism in (\ref{ShiodaIso}) depends on the choice of the base point $O \in X(\bar{k}(S))$.\\
\end{enumerate}
\end{remark}

\subsection{Examples}
\mbox{}
\vspace{5pt}

The K3 surfaces that we consider in Examples~\ref{ex:1} and \ref{ex:2} have Picard number at least 2. Indeed, in both cases the branch sextic admits a tritangent line that yields a section for the fibration.  Since the base of the fibration is rational, any section is a rational curve and thus yields an independent class in the Picard lattice. 

Let $L:=\big(\begin{smallmatrix}2&1\\1&-2 \end{smallmatrix}\big)$. By the discussion above, the K3 surfaces in Examples~\ref{ex:1} and \ref{ex:2} satisfy $L\subseteq \mathrm{NS}(X)$.

Our final example takes a different turn and considers a surface with Picard number 1, which fits the setting of Corollary \ref{cor3:intro}.

\begin{example}\label{ex:1}
We consider an $L$-polarized  K3 surface of degree 2 with $\rho= 2$.

Let $B$ be the smooth plane sextic with equation 

\[
B: x^5z+2x^4y^2+8x^4    z^2+x^3y^3+x^3yz^2+x^3z^3+x^2y^4+5x^2y^2z^2
\]
\[+7x^2z^4+xy^4z+xy^3z^2+6xz^5+y^5x+16y^4z^2=0.
\]
Let $X$ be the double cover of the plane branched along $B$. Then $X$ is a K3 surface with $\mathrm{NS}(X)=L$. Indeed, the class of a general line $\ell$ in $\mathbb{P}^2_{\mathbb{Q}}$ pulls back to a class $h$ in $X$ with $h^2=2$. We exhibit a class $m$ with $m^2=-2$, i.e., a rational curve, such that $m\cdot h =1$. The sextic $B$ admits a tritangent line, namely $x=0$. This line splits in the double cover in two classes of rational curves, $m$ and $m'$ in $\mathrm{NS}(X)$, such that $m\cdot h=m'\cdot h=1$ so $L \subseteq \mathrm{NS}(X)$ and, in particular, $\rho\geq 2$. In what follows, we show that equality holds, i.e., $\rho=2$. To conclude that $\mathrm{NS}(X)=L$, observe that $|\det L|=5$, and, in particular, it is square free, thus $\mathrm{NS}(X)$ is not an overlattice of $L$. 

We apply Corollary 2.3 in \cite{vanLuijkPic1}. More precisely, we show that for a prime of good reduction, the number of eigenvalues of the characteristic polynomial of Frobenius that are roots of unity is precisely 2. By Corollary 2.3 in \cite{vanLuijkPic1}, this gives $\rho\leq 2$.

All primes up to 100 are of good reduction, and, $p=19$ is the smallest prime that satisfies what we claimed above.
Indeed, in that case, the characteristic polynomial of Frobenius has the following irreducible factor of degree 20:
\[
\Phi_{20}(T)=T^{20} - 21T^{19} + 19T^{18} + 5776T^{17} - 164616T^{16} + 
        2345778T^{15} + 4952198T^{14} \]
        \[
        - 893871739T^{13} + 7150973912T^{12}  + 118884941287T^{11} -  2904189280011T^{10} + \]
         \[42917463804607T^9 + 
        931922071185752T^8 -42052983462257059T^7 + 
         84105966924514118T^6 + \]      \[14382120344091914178T^5-    364347048716995159176T^4+  4615062617081938682896T^3 + \]
         \[5480386857784802185939T^2 -2186674356256136072189661T+ 37589973457545958193355601.\]
        
The remaining factor corresponds to the two roots defined over the ground field, which stem from the classes of $h$ and $m$ described above, confirming the equality $\rho=2$.

In particular, $X$ does not admit a genus 1 fibration. Indeed, if there were a genus 1 fibration, there would exist a nef class $u$ with $u^2=0$, and $u\cdot m\geq 0$ in $\mathrm{NS}(X)$. We write $u=\alpha l+\beta m$, then $u^2=0$ implies that \begin{equation}\label{eq:1} 
2\alpha^2+2\alpha\beta -2\beta^2=0\
\end{equation}
Therefore $\beta\neq 0$ and $\alpha(\alpha+\beta) =\beta^2$. On the other hand, $u\cdot m\geq 0$, implies that 
\begin{equation}\label{eq:2}
\alpha -2\beta\geq 0, \text{ so } \alpha\geq 2\beta.
\end{equation}
Equations \ref{eq:1} and \ref{eq:2} yield 
\[
10\beta^2\leq 0,
\]
and hence $\beta=0$, which is not possible.  Therefore, there is no genus 1 fibration on $X$.

Let $P=(0:1:1) \notin B(\mathbb{Q})$ and let $\pi_P: \mathcal{X}_P \rightarrow \mathbb{P}^1_{\mathbb{Q}}$ be the genus 2 fibration induced by the pull-back of the pencil of lines through $P$, after two consecutive blow ups. Then $\pi_P$ admits a singular fiber with reducible components $m$ and $m'$, and two sections, namely the exceptional divisors above $(0:1:1:\pm 4)$. 

By the construction of $\mathcal{X}_P$, the group $\mathrm{NS}(\mathcal{X}_P)$ has rank 4. Since $x=0$ is the unique tritangent line to $B$ and $B$ is smooth, the unique reducible fiber of $\pi_P$ is $m+m'$. 

By the Shioda-Tate formula (\ref{ShiodaFormula}), the rank of the generic jacobian is $\operatorname{rank}(J(K))=1$, for $K=\mathbb{Q}(\mathbb{P}^1)$. 

We exhibit a singular saliently ramified multisection of genus 1 for $\pi$ defined over $\mathbb{Q}$. This puts in in position to apply Prop. \ref{infptrkjump} to conclude that there are infinitely many fibers of $\pi$ whose jacobian has rank at least 2, over $\mathbb{Q}$. 

The line $z=0$ is tangent to $B$ at the point $R=(1:0:0)$. On $X$ this corresponds to
\[
M:w^2=y^2x(2x^3+x^2y+xy^2+y^3),
\]

The normalization is an elliptic curve given by
\[E:w^2=x(2x^3+x^2y+xy^2+y^3)\]
Its Weierstrass form is 
\[ y^2 = x^3 + 864x + 81216,\] which has rank 1 over $\mathbb{Q}$.

We claim that it is saliently ramified. Indeed, the lines through $P$ and the intersection points of $\{z=0\}$ and $B$ different from $R$ correspond to smooth fibers of $\pi_P$, since each of them intersects $B$ in 6 distinct points (over $\bar{\mathbb{Q}}$). Thus, $M$ is a saliently ramified multisection of $\pi_P$ with $|M(\mathbb{Q})|=\infty$. We conclude that there are infinitely many $a \in \mathbb{Q}$ such that the jacobian of the hyperelliptic curve given by $w^2=F(x,y,y+ax)$ has rank at least 2.

\end{example}

\begin{example}\label{ex:2}

We consider $X$, the $U\oplus \langle -10\rangle$-polarized K3 surface from \cite{GarbagnatiSalgado}[Example 6.1], namely the minimal desingularization of the double cover of the plane branched over the plane sextic

\[
B: 2x^4y^2 + 8x^4z^2 + x^3y^3 + x^3yz^2 + x^3z^3 + x^2y^4 + \]
\[5x^2y^2z^2 + 7x^2z^4 + xy^5 + xy^4z + xy^3z^2 + 6xz^5 + 5y^4z^2=0.
\]

We exhibit a genus 2 pencil $\pi_P$ on $X$ that satisfies the following:

\begin{itemize}
\item[i)] $| \mathcal{R}(\pi_P,\mathbb{Q})|=\infty$;
\item[ii)] $\mathcal{R}(\pi_P,\mathbb{Q}(\sqrt{5})) \subset \mathbb{P}^1(\mathbb{Q}(\sqrt{5}))$ is not thin.
\end{itemize}

$X$ admits a unique genus 1 fibration, namely the pull-back of the pencil of lines through $R=(1:0:0)$, which moreover has a section over $\mathbb{Q}(\sqrt{5})$ and, by \cite{GarbagnatiSalgado} it has infinitely many fibers with positive Mordell--Weil rank over that field. These yield a family of bisections for the genus 2 fibration that we describe in what follows.

Let $\mathcal{X}_P$ be the blow up of $X$ at the point above $P=(0:0:1)$ and an infinitely near point. Then $\rho(\mathcal{X}_P)=5$. The pencil of lines through $P$, given by $y=tx$, yields a genus 2 fibration on $\pi_P:\mathcal{X}_P \rightarrow \mathbb{P}^1_k$ with equation
\[ w^2=x^6(2t^2+t^3+t^4+t^5)+ x^5zt^4 + x^4z^2(8+t+5t^2+t^3+5t^4) + x^3z^3  + 7x^2z^4  + 6xz^5,
\]
which admits a section over $\mathbb{Q}$, namely $[(x:z:w),t]=[(0:1:0),t]$.

We identify two reducible fibers: the pull-back of the line through $R$ and $P$ that admits two components, and the pull-back of the (degenerate) tritangent line $x=0$ that admits 3 components, namely the two lines above the tritangent and the first exceptional divisor of the blow up of $P$. %\textcolor{green}{There is another one. The infinitely near blowup gives a reducible fiber.} \textcolor{red}{The infinitely near blow up gives a component of a fiber, not a fiber. Fibers are pull-backs of lines. The infinitely near blow up will happen as a component above one of these lines, namely the tritangent $x=0$. So the fiber there has 3 components instead of 2 as I wrote}. 
By the Shioda--Tate formula  (\ref{ShiodaFormula}), 

\[
5=\rho(X)\geq  r +2+ 1+2,
\]

and hence $r=0$.

%First observe that $\pi$ does not admit further reducible fibers. Indeed, a reducible fiber corresponds to a line that intersects $C$ everywhere with even multiplicity, i.e., tritangent line, or that goes through a singularity of the sextic curve. These were described above, since $x=0$ is the only tritangent line to $C$.

We apply Theorem \ref{theorem: nonthinrkjumpgenus2} to conclude that, for $l=\mathbb{Q}(\sqrt{5})$, the set $\mathcal{R}(\pi_P,l)$ of fibers of rank at least 1 over $l$-points is not thin in $\mathbb{P}^1(l)$.

On the other hand, the pullback of the line $z=0$ to $X$ is the elliptic curve with equation
$$w^2=x(2x^3+x^2y+xy^2+y^3)$$
which has infinitely many $\mathbb{Q}$-points as seen in Example 1. Moreover, it is a saliently ramified bisection for $\pi_P$. A direct application of Proposition \ref{infptrkjump} yields $|\mathcal{R}(\pi_P, \mathbb{Q})|=\infty$.
\end{example}

\begin{example}
Let 
\begin{align*}
F(x,y,z):=\frac{7}{73}(&11x^5y +7x^5z + x^4y^2 +5x^4yz+7x^4z^2 +7x^3y^3 +10x^3y^2z \\
&+5x^3yz^2 +4x^3z^3+6x^2y^4 +5x^2y^3z +10x^2y^2z^2 +5x^2yz^3\\ 
&+5x^2z^4+11xy^5 +5xy^3z^2 +12xz^5+9y^6 +5y^4z^2 +10y^2z^4 +4z^6). 
\end{align*}
We consider the K3 surface $X$ given by
\begin{align*}
w^2=F(x,y,z). 
\end{align*}
It is in the family of degree two K3 surfaces of geometric Picard rank 1 constructed in \cite{JuliaRatPtK3deg2}. The surface $X$ is geometrically isomorphic to the surface in \cite[Corollary~30]{K3Pic1deg2}. It is showed in \cite[Theorem~5.3]{JuliaRatPtK3deg2} that the tangent line 
$$T_Q:11y+7z=0$$ 
to the branching curve $B=V(F) \subset \mathbb{P}^2_k$ at $Q=[1:0:0] \in B(\mathbb{Q})$ pulls back to a curve $C=f^{-1}(T_Q)$ whose normalization $\hat{C}$ is given by
$$ w^2 = \frac{16771780}{1226911} y^4 - \frac{1540220}{175273} xy^3 + \frac{81451}{25039}x^2y^2 - \frac{4078}
{3577}x^3y +x^4.$$
The curve $\hat{C}$ has infinitely many rational points. Let $P=[0:1:0]$ and consider the genus two fibration 
$$\pi_P : \mathcal{X}_P \rightarrow \mathbb{P}^1_{\mathbb{Q}}$$
given by
\begin{align}
    w^2=F(x,y,ty)
\end{align}
over $\mathbb{Q}(t)$. Using MAGMA, we check that $C$ is a saliently ramified multisection of $\pi_P : \mathcal{X}_P \rightarrow \mathbb{P}^1_{\mathbb{Q}}$, so there are infinitely many $a \in \mathbb{Q}$ such that the Jacobian of the curve given by $w^2=F(x,y,ay)$ has rank at least 1 over $\mathbb{Q}$. 

The two exceptional divisors of the blowup $\nu_P: \mathcal{X}_P \rightarrow \mathbb{P}^1_{k_P}$ are defined over the field $\mathbb{Q}(\sqrt{63/73})$ and yield two sections of $\pi_P$. By Theorem \ref{theorem: shioda formula}, the jacobian of $\pi_P$ has generic rank $\geq 1$ over $\mathbb{Q}(\sqrt{63/73})$. Then, for infinitely many $a \in \mathbb{Q}$ the rank of the jacobian of the curve $w^2=F(x,y,ay)$ over $\mathbb{Q}(\sqrt{63/73})$ is at least 2.

\end{example}

\section*{Acknowledgments}
This work was supported by an NWO Open Competition ENW– XL grant (project ”Rational
points: new dimensions”- OCENW.XL21.XL21.011). We thank Marc Hindry pointing out Theorem \ref{djamanenkomanin} as a tool in the proof of Theorem \ref{specialization}. We also thank Ariyan Javanpeykar for a helpful exchange on fibrations and the Hilbert property.
\printbibliography

\vspace{30pt}

\noindent\textsc{Bernoulli Institute, University of Groningen, The Netherlands}
\vspace{3pt}

\noindent\textit{Email address:} \texttt{a.arriola.corpion@rug.nl}\\
\textit{URL:} \texttt{https://www.rug.nl/staff/a.arriola.corpion/}

\vspace{3pt}

\noindent\textit{Email address:} \texttt{c.salgado@rug.nl}\\
\textit{URL:} \texttt{https://www.math.rug.nl/algebra/Main/salgado}

\end{document}